\documentclass{amsart}
\usepackage{xypic}
\usepackage{amssymb, amsmath, amsthm}
\usepackage[usenames]{color}

\newcommand{\G}{\Gamma}

\newcommand{\bs}{\backslash}
\newcommand{\Aut}{\mathop{\mathrm{Aut}}}
\newcommand{\Stab}{\mathop{\mathrm{Stab}}}

\def\Aut{\mathrm{Aut}}

\def\Cnn{\mathrm{Comm}_{\leq n}(\Gamma,G)}
\def\Cnnp{\mathrm{Comm}_{\leq n}(\Gamma',G)}
\def\Cnnt{\mathrm{Comm}_{\leq tn}(\Gamma,G)}
\def\Cnnpt{\mathrm{Comm}_{\leq tn}(\Gamma',G)}
\def\Cnng{\mathrm{Comm}_{\leq n}(\Gamma^g,G)}

\def\c{c}

\newcommand{\al}{\alpha}
\newcommand{\be}{\beta}
\newcommand{\ga}{\gamma}
\newcommand{\de}{\delta}

\newcommand{\la}{\lambda}
\newcommand{\ze}{Z}

\newcommand{\Ga}{\Gamma}
\newcommand{\De}{\Delta}

\newcommand{\hR}{\widehat{R}}

\newcommand{\fg}{\mathfrak g}

\newcommand{\tg}{\widetilde{g}}

\newcommand{\bA}{\mathbb A}

\newcommand{\bF}{\mathbb F}
\newcommand{\bN}{\mathbb N}
\newcommand{\bQ}{\mathbb Q}
\newcommand{\bR}{\mathbb R}
\newcommand{\bZ}{\mathbb Z}

\newcommand{\bmat}{\left(\begin{matrix}}
\newcommand{\emat}{\end{matrix}\right)}

\newcommand{\hZ}{\widehat{\mathbb Z}}
\newcommand{\A}{\mathbb{A}^f}

\DeclareMathOperator{\Comm}{Comm}
\DeclareMathOperator{\val}{val}
\DeclareMathOperator{\vol}{vol}
\DeclareMathOperator{\GL}{GL}

\DeclareMathOperator{\PGL}{PGL}

\DeclareMathOperator{\M}{M}
\DeclareMathOperator{\Ad}{Ad}
\def\vol {\text{vol}}

\newtheorem{theo}{\bf{Theorem}}[section]
\newtheorem{thm}[theo]{Theorem}
\newtheorem{lem}[theo]{Lemma}

\newtheorem{defi}[theo]{Definition}
\newtheorem{defn}[theo]{Definition}
\newtheorem{prop}[theo]{Proposition}

\newtheorem{prob}[theo]{Problem}

\begin{document}
\title{On Commensurizer Growth}
\author{Nir Avni}
\address{Department of Mathematics, Harvard University, Cambridge, MA}
\author{Seonhee Lim}
\address{Department of Mathematical Sciences, Seoul National University, Seoul 151-747, South Korea}
\email{slim@snu.ac.kr}
\author{Eran Nevo}
\address{Department of Mathematics\\
    Cornell University\\
    Ithaca, NY 14853, USA}
\email{eranevo@math.cornell.edu}
\thanks{The first author was partially supported by NSF Award DMS-0901638. The third author was partially supported by an NSF Award DMS-0757828.}
\maketitle
\begin{abstract}
We study new asymptotic invariant of a pair consisting of a group and a subgroup, which we call Commensurizer growth. We compute the commensurizer growth for several examples, concentrating mainly on the case of a locally compact topological group and a lattice inside it.
\end{abstract}

%\maketitle
\section{Introduction}
Consider a group $G$ and a subgroup $A$ of $G$. For an element $g\in G$, we denote the conjugation-by-$g$ map as $x\mapsto x^g=g^{-1}xg$. We say that an element $g\in G$ commensurates $A$ if $A\cap A^g$ has finite index in both $A$ and $A^g$. The set of elements in $G$ that commensurate $A$ is called the commensurability group or the {\it commensurizer} of $A$ in $G$; we denote it by $\Comm(A,G)$.

Classical results on the commensurizer include superrigidity of commensurizer, proved by Margulis \cite{M} for lattices in semisimple Lie groups of rank $>1$ and by Lubotzky-Mozes-Zimmer \cite{LMZ} for tree lattices.
If $A$ is a uniform lattice, i.e. a lattice of compact quotient, the commensurizer is dense in the ambient group $G$ in both cases (when $G$ is the automorphism group of a tree, this was first shown in \cite{BaKu}).

The set $\Comm(A,G)$ is naturally filtered according to the index of $A\cap A^g$ in $A$. More precisely, define the $n$'th commensurizer to be
\[
\Comm_n(A,G)=\{g\in \Comm(A,G)\;|\; [A:A\cap A^g]=n\}.
\]
The normalizer of $A$ in $G$, which we denote by $N_G(A)$, acts on the left on the sets $\Comm_n(A,G)$. We denote the size of the quotient $N_G(A)\bs\Comm_n(A,G)$ by $c_n(A,G)$. By definition, $\Comm_1(A,G)=N_G(A)$, and so $c_1(A,G)=1$. In general, the numbers $c_n(A,G)$ might be infinite, but we will soon restrict to pairs $(A,G)$ for which $c_n(A,G)$ are finite for every $n$. The asymptotic behavior of the sequence $c_n(A,G)$ is what we call the {\it commensurizer growth} of the pair $(A,G)$. We will usually phrase our results using the sequence
\[
c_{\leq n}(A,G)=c_1(A,G)+\ldots+c_n(A,G),
\]
counting the set of elements $g\in \Comm(A,G)$, up to $N_G(A)$, such that the index of $A\cap A^g$ in $A$ is at most $n$.

In this paper we study the commensurizer growth for several classes of pairs of groups, mostly for pairs $(A,G)$ such that $G$ is a topological group and $A$ is a lattice in $G$.
The two main examples are the case of uniform lattices in a Lie group $PGL_2(F)$ over a non-archimedean local field $F$ and the case of uniform lattices in the automorphism group of a tree. Although in both cases the commensurizer is dense in the ambient group, we will show that the commensurizer growths are very different : polynomial v.s. exponential.

The case where $G$ is abelian is trivial. One of the first non-trivial cases among Lie groups is the pair $(H(\bZ),H(\bR))$, where $H$ is the three-dimensional Heisenberg group.

\begin{thm}\label{thm:HeisenbergGrowth}
Let $H$ be the three-dimensional Heisenberg group. Then the sequence $c_{\leq n}(H(\bZ),H(\bR))$ grows cubically. More precisely
\[
\lim_{n\to\infty}\frac{c_{\leq n}(H(\bZ),H(\bR))}{n^3}=\frac{1}{3\zeta(3)}=0.277\ldots
\]
\end{thm}
An even smaller commensurizer growth is obtained by the pair $(\PGL_2(\bZ),\PGL_2(\bR))$, for which the growth is quadratic. More generally, we prove

\begin{thm}\label{thm:PGL2Growth}
Let $F$ be a local field, and let $\Ga$ be an arithmetic lattice in the group $\PGL_2(F)$. Then the sequence $c_{\leq n}(\Ga,\PGL_2(F))$ grows quadratically. More precisely, the sequence $c_{\leq n}(\Ga,\PGL_2(F))/n^2$ is bounded away from $0$ and infinity, as $n$ tends to infinity.
\end{thm}
It is well known that if $F$ is a non-archimedian local field, then the group $\PGL_2(F)$ acts faithfully and transitively by isometries on a $(q+1)$-regular tree, where $q$ is the size of the residue field of $F$ \cite{S}. Denote the group of isometries of the $d$-regular tree by $\Aut(T_d)$. The action gives an embedding of $\PGL_2(F)$ in $\Aut(T_{q+1})$, which is uniform (i.e. the quotient $\PGL_2(F)\backslash \Aut(T_{q+1})$ is compact). This means that if $\Ga$ is a uniform lattice in $\PGL_2(F)$ (recall that such lattices always exist), then it is also a uniform lattice in $\Aut(T_{q+1})$, thus one can consider the commensurizer growth of $\Ga$ in $\Aut(T_{q+1})$ as well. It turns out that the growth in $\Aut(T_{q+1})$ is much bigger than the growth in $\PGL_2(F)$:
\begin{theo}\label{thm:upper bound uniform lattice}\label{thm:exact_bound}
Let $T$ be a uniform tree, $G=\Aut(T)$ and $\Gamma$ a
uniform lattice in $G$. Assume that every vertex in $T$ has at least $3$ neighbors.

Then there exist positive constants $c_1({\G})$ and $c_2({\G})$ such that
for any $n$ large enough,
$$2^{c_2({\G})n \lg(n)}\leq c_{\leq n}(\Ga, G) \leq 2^{c_1({\G})n \lg(n)}.$$
\end{theo}

We show that, for a general pair of groups, the commensurizer growth can be arbitrarily big:

\begin{theo}\label{thm:manyGrowthFunc}
Let $f:\bN\rightarrow \bN$. Then there exists a pair $\G<G$ of groups   such that $c_{\leq n}(\Ga, G)\geq f(n)$ for all $n \in \bN$.
\end{theo}

We may still hope for a positive answer to the following problem.
\begin{prob}
Is there a function $f:\bN\rightarrow \bN$ such that for any
lattice $\G$ in a finitely generated group $G$, its commensurizer growth function satisfies $c_n(\G, G)<f(n)$ for any $n$?
\end{prob}

Outline: in Section \ref{sec:pre} we establish general facts about the commensurizer, to be used throughout the paper. In Section \ref{sec:Heisenberg} we prove Theorem \ref{thm:HeisenbergGrowth}, in Section \ref{sec:PGL2} we prove Theorem \ref{thm:PGL2Growth}, in Section \ref{sec:countRecolorings} we prove Theorem \ref{thm:upper bound uniform lattice} and in Section \ref{sec:arbitraryGrowth} we prove Theorem \ref{thm:manyGrowthFunc}.
%We end Section \ref{sec:arbitraryGrowth} with questions concerning quantitative versions of the density of $\Comm(\G,G)$ in $G$, using commensurizer growth, and uniform distribution.

%%%%%%%%%%%%%%%%%%%%%%%%%%%%%%%%%

%%%%%%%%%%% Preliminaries %%%%%%%%%%%%%%%

%%%%%%%%%%%%%%%%%%%%%%%%%%%%%%%%%

\section{Preliminaries}\label{sec:pre}

Let $G$ be a group and $A$ be a subgroup of $G$. For $g\in G$, let $\chi_{A,G}(g)=[A:A\cap A^g]$. We start by showing that the commensurizer growth of a pair $(A,G)$ is unchanged when $A$ is replaced by a finite index subgroup.

\begin{lem}\label{lem:finite groups}
Let $G$ be a finite group. If $[G:B]=t$ and $[G:C]=k$ then $[G:C\cap B]\leq kt$.
\end{lem}
\begin{proof} Consider the map $f: B \times C \to G, f((b,c))=bc.$ We claim that the fibers all have cardinality $|B\cap C|$. For $g=b_0c_0 \in BC$, $f^{-1}(g) = \{ (b,c) : bc =g \}.$ Since $c=b^{-1}g =b^{-1}b_0c_0 \in C$, we have $b \in b_0C$. Therefore $|f^{-1}(g)|=|\{(b,b^{-1}g): b \in B \cap b_0C\}|=|b_0B\cap b_0C|=|B \cap C|$.
It follows that $|B||C|=|B \times C| \leq |B \cap C||Im(f)|\leq|B\cap C||G| = |B \cap C||B|t$, therefore $[C: B \cap C] \leq t$.\end{proof}

The lemma below is well known. For the reader's convenience, we provide a proof.
\begin{lem}\label{lem:normal sgp}
Let $G$ be finitely generated, $[G:C]=t$ and $[G:A]=s$. Then there exists a  normal subgroup $N\triangleleft G$ of finite index such that for any $H<G$ with $[G:H]\leq \rm{max}(t,s)$, $N\subseteq H$ holds. In particular, $N< C\cap A$.
\end{lem}
\begin{proof}
Let $n\in \mathbb{N}$ and assume that $G$ is generated by a set of size $l$. Then $|\hom(G,S_n)|\leq (n!)^l<\infty$, where $\hom(G,S_n)$ is the group of homomorphisms from $G$ to the permutation group on $n$ elements, $S_n$. Let $\psi\in \hom(G, S_n^{|\hom(G,S_n)|})$ be defined by $(\psi(g))_{\rho}=\rho(g)$ for all $g\in G$ and $\rho\in \hom(G,S_n)$. Let $N=\ker(\psi)$ be the kernel of $\psi$. Thus $N$ is a normal subgroup of finite index in $G$.
If $H$ is a subgroup of $G$ of index $n$ then $S_n$ acts on $G/H$ (by permuting the cosets). Let $\rho'\in \hom(G,S_n)$ be the homomorphism defined by $\rho'(g)$ being the permutation acting on $G/H$ by multiplication by $g$. Thus
$H=\{g\in G: \rho'(g)(1H)=1H\} \supseteq N$.
By embedding $S_m\hookrightarrow S_n$ as a subgroup for $1\leq m\leq n$, the same reasoning as above shows that if $[G:H]=m<n$ then $N<H$. For $n=\max(s,t)$ we get $N$ as required.
\end{proof}

The following lemma relates the commensurizer growths of commensurable subgroups $\G, \G'$ of $G$. Let us denote $\Cnn = \underset{l \leq n}{\cup} \Comm_l(\G, G)$.

\begin{prop}\label{prop:IndexGapForCommensurales}
Let $G$ be finitely generated group and $\G'\subseteq \G$ be two subgroups in $G$, where $[\G : \G']=t$. Then for any $n$, $\Cnn\subseteq \Cnnpt$ and also $\Cnnp\subseteq \Cnnt$. Equivalently, for each $g\in G$
\[
\frac{1}{t}\chi_{\Ga,G}(g)\leq\chi_{\Ga',G}(g)\leq t\chi_{\Ga,G}(g)
\]
\end{prop}
\begin{proof}
Let $g\in \Comm_n(\G,G)$, thus $[\G: \G\cap \G^g]=n$, $[\G^g: \G\cap \G^g]=k<\infty$ and recall that $[\G:\G']=t$. Note that also $[\G^g:\G'^g]=t$.
Let $N\triangleleft \G^g$ be as guaranteed in Lemma \ref{lem:normal sgp} w.r.t. $\G'^g$ and $\G\cap \G^g$. If $N<H<K<\G^g$ then $[K:H]=[K/N:H/N]$, thus by Lemma \ref{lem:finite groups} we get $[\G^g:\G\cap \G'^g]\leq kt$, hence $[\G\cap \G^g: \G\cap \G'^g]\leq t$. We get $[\G:\G\cap \G'^g]\leq nt$, also $[\G:\G']=t$, hence by applying Lemmata \ref{lem:finite groups} and \ref{lem:normal sgp} for these groups we obtain for $\G'\cap (\G\cap \G'^g) = \G'\cap \G'^g$ that $[\G':\G'\cap \G'^g]\leq tn$.

The second containment is easy: $[\G': \G'\cap \G'^g]=n$ implies $[\G: \G'\cap \G'^g]= tn$. As $\G'\cap \G'^g \subseteq \G\cap \G^g$ we get $[\G:\G\cap \G^g]\leq tn$.
\end{proof}

\begin{lem} \label{lem:chi.product} Assume that $G$ is a (topologically) finitely generated unimodular group and that $A$ is either an open compact subgroup of $G$ or a lattice in $G$. Then, \begin{enumerate}
\item $\chi(g)=\chi(g^{-1})$.
\item $\chi(gh)\leq\chi(g)\chi(h)$.
\end{enumerate}
\end{lem}

\begin{proof} The first claim is equivalent to $[A:A\cap A^g]=[A^g:A\cap A^g]$. If $\la$ is a Haar measure on $G$ then,
for $A$ an open compact subgroup of $G$,
$[A:A\cap A^g]=\la(A)/\la(A\cap A^g)=\la(A^g)/\la(A\cap A^g)=[A^g:A\cap A^g]$.

If $A$ is a lattice in $G$ then the volumes equality $\vol(G/A)=\vol(G/A^g)$ holds, hence
$[A:A\cap A^g]=\vol(G/A)/\vol(G/A\cap A^g)=\vol(G/A^g)/\vol(G/A\cap A^g)=[A^g:A\cap A^g]$.

To prove the second claim, note that the index of $A^h\cap A^{gh}$ in $A^h$ is equal to $\chi(g)$. By the first claim, the index of $A\cap A^h$ in $A^h$ is $\chi(h)$. Therefore, by Lemmata \ref{lem:normal sgp} and \ref{lem:finite groups}, the index of $A\cap A^h\cap A^{gh}$ in $A^h$ is at most $\chi(g)\chi(h)$.
As by part (1) $[A^h:A\cap A^h]=\chi(h)$, we get $[A\cap A^h: A\cap A^h\cap A^{gh}]\leq \chi(g)$, hence $[A: A\cap A^h\cap A^{gh}]\leq\chi(g)\chi(h)$ and part (2) follows.
\end{proof}
Next we discuss the relation between $\Comm_n$ and products.
\begin{lem} \label{lem:Comm.prod}

\begin{enumerate}
\item Let $A\subset G$ and $B\subset H$ be groups. Then $\Comm_n(A\times B,G\times H)=\bigsqcup_a \Comm_a(A,G)\times\Comm_{n/a}(B,H)$.
\item Let $A_i\subset G_i$ be groups, and let $\prod^{'}G_i$ denote the restricted product of the $G_i$'s relative to the $A_i$'s (i.e. the set of elements of which all but finitely many entries are in $A_i$). Then
\[
\Comm_n(\prod A_i,\prod^{'} G_i)=\bigsqcup \prod_i \Comm_{a_i}(A_i,G_i),
\]
where the disjoint union is over the set of sequences $(a_i)\in\bN^\infty$ such that $a_i=1$ for all but finitely many $i$'s and $\prod_i a_i=n$.
\end{enumerate}
\end{lem}

\begin{proof} (1) is clear. For every finite set $S$ of indices, let
\[
G_S=\prod_{i\in S} G_i \times \prod_{i\notin S}A_i.
\]
By (1), we have that
\[
\Comm_n(\prod A_i,G_S)=\sqcup \prod_{i\in S}\Comm_{a_i}(A_i,G_i)
\]
where the union is over the sequences $(a_i)\in\bN^S$ such that $\prod a_i=n$.
Since
\[
\prod'G_i = \lim_{\to}G_S,
\]
we get that
\[
\Comm_n(\prod A_i,\prod^{'}G_i)=\lim_{\to}\Comm_n(\prod A_i,G_S),
\]
which implies (2).
\end{proof}

A snazzier way to formulate the last lemma is by using generating functions.

\begin{defn} Let $A\subset G$ be groups. Define the commensurizer zeta function of $(A,G)$ to be
\[
\ze_{A,G}(s)=\sum_n c_n(A,G)\cdot n^{-s}=\sum_{g\in N_G(A)\bs\Comm(A,G)}\chi_{A,G}(g)^{-s}.
\]

If the function $\ze_{A,G}(s)$ converges somewhere, then its domain of convergence is a half plane of the form $\{s|\Re(s)>\al\}$. This $\al$ is called the abscissa of convergence of $\ze_{A,G}(s)$; we denote it by $\alpha_{A,G}$.
\end{defn}

In terms of the last definition, Lemma \ref{lem:Comm.prod}(1) states that
\[
Z_{A\times B,G\times H}(s)=Z_{A,G}(s)Z_{B,H}(s)
\]
and, as $N_{\prod'G_i}\prod A_i$ is the direct limit of the $N_{G_S}\prod A_i$, Lemma \ref{lem:Comm.prod}(2) states that
\[
Z_{\prod A_i,\prod'G_i}(s)=\prod Z_{A_i,G_i}(s).
\]

Such infinite products arise from Adelic groups. Here is an example. Let $\hZ=\prod\bZ_p$ be the pro-finite completion of the integers, and let $\A=\hZ\otimes_\bZ\bQ$ be the ring of finite Adeles. Note that $\hZ\cap\bQ=\bZ$.

\begin{defi} Let $G$ be an algebraic group. We say that $G$ satisfies the strong approximation property if $G(\bQ)$ is dense in $G(\A)$.
\end{defi}

\begin{lem} \label{lem:Z.to.Adeles} Let $G$ be an algebraic group defined over $\bQ$ that satisfies the strong approximation property. Then
$$c_{n}(G(\bZ),G(\bQ))=c_{n}(G(\hZ),G(\A)).$$
\end{lem}

\begin{proof} Let $g$ be a rational matrix. Let us see that $[G(\hZ):G(\hZ)\cap G(\hZ)^g]=[G(\bZ):G(\bZ)\cap G(\bZ)^g]$. Indeed, every coset of $G(\hZ)\cap G(\hZ)^g$ is open and hence contains a rational matrix $a$, but since $a\in G(\hZ)$, it follows that $a\in G(\bZ)$. In addition, if $a,b\in G(\bZ)$ are representatives for different $G(\bZ)\cap G(\bZ)^g$ cosets, and assume that the cosets $aG(\hZ)\cap G(\hZ)^g$ and $bG(\hZ)\cap G(\hZ)^g$ intersect, then their intersection contains a rational matrix, which has to be integral, a contradiction.

Let $N=N_{G(\bQ)}G(\bZ)$ and $\widehat{N}=N_{G(\A)}G(\hZ)$. Every coset $g \widehat{N}$ of $\widehat{N}$ where $g \in \Comm_n(G(\hZ), G(\A))$ contains a rational matrix, since $\widehat{N}$ is open in $G(\A)$. Given a rational matrix $g$ in $\Comm_n(G(\hZ),G(\A))$, the argument above shows that $g\in\Comm_n(G(\bZ),G(\bQ))$. Finally, if $g\in G(\bQ)\cap\widehat{N}$, then for every $h\in G(\bZ)$, the conjugation $g^h=h^{-1}gh$ is both in $G(\bQ)$ and in $G(\hZ)$. Therefore $g\in N$. All this means that $gN\mapsto g\widehat{N}$ is a well-defined bijection between $N \backslash \Comm_n(G(\bZ),G(\bQ))$ and $\widehat{N}\backslash \Comm_n(G(\hZ),G(\A))$.
\end{proof}

Using Lemma \ref{lem:Z.to.Adeles} and Lemma \ref{lem:Comm.prod}, we get

\begin{prop} \label{prop:Euler.product} Let $G$ be an algebraic group that satisfies the strong approximation property. Then
\[
Z_{G(\bZ),G(\bQ)}(s)=\prod_p Z_{G(\bZ_p),G(\bQ_p)}(s).
\]
\end{prop}

%%%%%%%%%%%%%%%%%%%%%%%%%%%%%%%%%%%%

%%%%%%%%%%% section Heisenberg group %%%%%%%%%%%

%%%%%%%%%%%%%%%%%%%%%%%%%%%%%%%%%%%%%

\section{The Heisenberg Group}\label{sec:Heisenberg}
Let $G$ be the three-dimensional Heisenberg group. Recall that as sets $G(\bZ)=\bZ^2\times\bZ$, and the conjugation is
\[
(u,\be)^{-1} (v,\al) (u,\be)=(v,\al+B(u,v)),
\]
where $B$ is the symplectic bilinear form corresponding to the matrix $\bmat 0&1\\-1&0\emat$.

The group $G$ satisfies the strong approximation property by the strong approximation theorem (\cite{P}, section 7.4). Moreover, the commensurizer of $G(\bZ)$ in $G(\bR)$ is equal to $\Comm(G(\bZ),G(\bQ))$ times the center of $G(\bR)$. Consequently, $$N_{G(\bR)}G(\bZ)\bs\Comm(G(\bZ),G(\bR))= N_{G(\bQ)}G(\bZ)\bs\Comm(G(\bZ),G(\bQ)),$$ thus $Z_{G(\bZ),G(\bR)}=\prod_p Z_{G(\bZ_p),G(\bQ_p)}$ by Proposition \ref{prop:Euler.product}.

Fix $p$, and let $(a,b)\in\bZ_p^2$ such that $\val(a)\leq\val(b)$. Then there is $\al\in\bZ_p$ such that $b=a\al$, and so $B((x,y),(a,b))=xb-ya=a(x\al+y)\in\bZ_p$ if and only if $x\al+y\in a^{-1}\bZ_p$. If $\val(a)\geq0$ then this is satisfied for all $(x,y)\in\bZ_p^2$. If $\val(a)<0$, then for every $x$, the set of $y$'s for which this holds has measure $p^{\val(a)}$. Hence
\[
\mu(G(\bZ_p)\cap G(\bZ_p)^{((a,b),\al)})=\mu\{(x,y)\in\bZ_p^2|B((x,y),(a,b))\in\bZ_p\}=\min\{p^{val(a)},p^{\val(b)},1\}
\]
and so
\[
[G(\bZ_p):G(\bZ_p)\cap G(\bZ_p)^{((a,b),\al)}]=\max\{p^{-\val(a)},p^{-\val(b)},1\}.
\]

We have that $N=N_{G(\bQ_p)}G(\bZ_p)$ is the product of $G(\bZ_p)$ and the center of $G(\bQ_p)$. Therefore, the set $N \backslash G(\bQ_p)$ is parameterized by $(\bQ_p/\bZ_p)^2$, where to the point $(a,b)\in(\bQ_p/\bZ_p)^2$ corresponds the coset
\[
\left\{\bmat1&x&z\\&1&y\\&&1\emat : x=a,y=b\textrm{ (mod $\bZ_p$)}\right\}.
\]

\[
Z_{G(\bZ_p),G(\bQ_p)}(s)=\sum_{a,b\in\bQ_p/\bZ_p}\left(\max\{p^{-\val(a)},p^{-val(b)}\}\right)^{-s}=
\]
\[
=\sum_{m,n\in\bN}|\{a\in\bQ_p/\bZ_p|\val(a)=m\}||\{a\in\bQ_p/\bZ_p|\val(a)=n\}|\left(\max\{p^m,p^n\}\right)^{-s}.
\]
Dividing the sum to $\{(0,0)\}$, the set $\{(m,0)\}_{m\geq1}\cup\{(0,n)\}_{n\geq1}$, the set $\{(m,m)\}_{m\geq1}$, and the set $\{(m,n)|1\leq m<n\}\cup\{(m,n)|1\leq n<m\}$, we get that
\begin{align*}
Z_{G(\bZ_p),G(\bQ_p)}(s)= &1+2\sum_{n=1}^\infty(p-1)p^{n-1}p^{-ns}+\sum_{m=1}^\infty(p-1)^2p^{2m-2}p^{-ms}\\
&+2\sum_{m=1}^\infty\sum_{n=m+1}^\infty(p-1)^2p^{m+n-2}p^{-ns}\\
=&1+2\frac{p-1}{p}\frac{p^{1-s}}{1-p^{1-s}}+\frac{(p-1)^2}{p^2}\frac{p^{2-s}}{1-p^{2-s}}+2\frac{(p-1)^2}{p^2}\frac{p^{1-s}}{1-p^{1-s}}\frac{p^{2-s}}{1-p^{2-s}}\\
=&\frac{1-p^{-s}}{1-p^{2-s}}
\end{align*}
There are two consequences to the computation above:
\begin{itemize}
%\item $Z_{G(\bZ_p),G(\bQ_p)}(s)$ is a rational function in $p^{-s}$. This means that the sequence $c_{p^n}(G(\bZ_p),G(\bQ_p))$ satisfies a linear recursion.
\item
$c_{p^n}(G(\bZ_p),G(\bQ_p))= p^{2n}(1-p^{-2})$.
\item $\ze_{G(\bZ),G(\bR)}(s)=\ze_{G(\bZ),G(\bQ)}=
\prod_p Z_{G(\bZ_p),G(\bQ_p)}=
\zeta(s-2)/\zeta(s)$, where $\zeta(s)$ is the Riemann zeta function.

Using standard Tauberian theorems (see, for example \cite[Theorem 4.20]{dSG}), we get
\[
\lim_{n\to\infty}\frac{c_{\leq n}(G(\bZ),G(\bR))}{n^{3}}=\frac{1}{3\zeta(3)},
\]
proving Theorem \ref{thm:HeisenbergGrowth}.
\end{itemize}

%%%%%%%%%%%%%%%%%%%%%%%%%%%%%%%%%

%%%%%%%%% Arithmetic Lattices in PGL_2(Z) %%%%%%%%%%

%%%%%%%%%%%%%%%%%%%%%%%%%%%%%%%%%%

\section{Arithmetic Lattices in $\PGL_2$}\label{sec:PGL2}

Let $F$ be a local field. We recall the construction of arithmetic lattices in $\PGL_2(F)$ (see \cite{P} for the general construction). Any such lattice is determined by a global field, $k$, and a form, $G$, of $\PGL_2$ over $k$. The global field $k$ is, by definition, either a finite extension of the field of rational numbers or the function field of an algebraic curve over a finite field; it should have a valuation, $v$, such that the completion of $k$ with respect to $v$, which we denote by $k_v$, is equal to $F$. Fix such a valuation $v$, and let $R$ be the ring of $v$-integers,
\[
R=\{x\in k | \textrm{$\val_w(x)\geq0$ for all non-archimedian valuations $w$ different from $v$}\}.
\]
Denote the pro-finite completion of $R$ by $\hR$; it is the product of all completions $R_w$ of $R$, where $w$ is taken from the non-archimedian valuations of $k$ that are different from $v$. Finally, let $\bA$ be the restricted product of the completions $k_w$ relative to $R_w$, taken over all non-archimedian valuations of $k$ that are different from $v$.

The form $G$ must be either the algebraic group $\PGL_2$, or the group $\PGL_1(D)$, where $D$ is a quaternion algebra over $k$. Since we are interested in lattices of $\PGL_2(F)$, we assume that $G$ splits at $v$---i.e. that $G(F)$ is isomorphic to $\PGL_2(F)$---and that $G$ does not split over all archimedian primes different from $v$. It is well known that, under these assumptions, $G(R)$ is a lattice in $G(k_v)=\PGL_2(F)$. If $G=\PGL_1(D)$, then this lattice is uniform. The group $G(R)$ is the arithmetic lattice\footnote{More accurately, an arithmetic lattice is a subgroup of $G(k_v)$ that is commensurable to $G(R)$. However, by Proposition \ref{prop:IndexGapForCommensurales}, the commensurizer growth does not change after passing to a commensurable group.}.

The following is well-known. For the reader's convenience, we supply a proof.
\begin{lem} The commensurizer of $G(R)$ in $G(k_v)$ is $G(k)$.
\end{lem}

\begin{proof} We prove the claim for $G=\PGL_1(D)$; the case $G=\PGL_2$ can be proved similarly. We first make a more precise statement. Let $\fg(k_v)$ be the Lie algebra of $G(k_v)$. We can identify $\fg(k_v)$ with the subspace $D^1\subset D(k_v)=\M_2(k_v)$ of elements with 0 trace. Note that $D^1$ is defined over $k$. Pick a basis of $D^1$ consisting of elements of $D(k)$, and denote the $k$-span of this basis by $D^1(k)$. The homomorphism $G(k_v)\to \Aut(\fg(k_v))$ given by $g\mapsto\Ad(g)$ (which is called the adjoint representation) is an isomorphism, as the kernel is the center of the group which is trivial in this case (see Appendix of \cite{Mo} for example). The claim in the proposition is that the image of $\Comm(G(R),G(k_v))$ is the set  of endomorphisms that preserve $D^1(k)$.

Indeed, let $g\in G(k_v)=\PGL_2(k_v)$ be in the commensurizer, and let $\De=G(R)\cap G(R)^g$. Denote the pre-image of $\De$ in $D(R)^\times$ by $\widetilde{\De}$. Since $\De$ is of finite index in $G(R)$, we get that $\widetilde{\De}$ is of finite index in $D(R)^\times$, and, in particular, that it spans $D(k)$. For every $\de\in\De$ choose lifts $\tg\in\GL_2(k_v)$ of $g$ and $\widetilde{\de}\in D(R)$ of $\de$. By assumption, $\widetilde{\de}^{\tg} \in D(R)^\times\cdot(k_v^\times I)$. Taking traces of both sides, we get that $\widetilde{\de}^{\tg} \in D(k)^\times$. Since $\widetilde{\de}$ was arbitrary and $\widetilde{\De}$ spans, we get that $\Ad(\tg)$ preserves $D(k)$ and therefore that $\Ad(g)$ preserves $D^1(k)$.
\end{proof}

For the proofs that follow, we need to slightly extend the notation. If $A\subset G$ are groups and $X\subset G$ is any subset containing $N_G(A)$, denote the collection of elements $g\in X$ such that $[A:A\cap A^g]=n$ by $\Comm_n(A,X)$. Similarly, one can define $\chi_{A,X}(g)$, $Z_{A,X}(s)$, and $\al_{A,X}$.

\begin{lem} \label{lem:R.k.hR.bA} For every $g\in G(k)$, $\chi_{G(R),G(k)}(g)=\chi_{G(\hR),G(\bA)}(g)$.
\end{lem}

\begin{proof} Suppose that $g_i$ are coset representatives for $G(R)\cap G(R)^g$ in $G(R)$. It is enough to show that they are also coset representatives for $G(\hR)\cap G(\hR)^g$ in $G(\hR)$. To show that they are disjoint, if $g_iG(\hR)\cap G(\hR)^g=g_jG(\hR)\cap G(\hR)^g$, then $g_j^{-1}g_i\in G(\hR)^g$, or $(g_j^{-1}g_i)^{g^{-1}}\in G(\hR)$. Since $(g_j^{-1}g_i)^{g^{-1}}\in G(k)$, we get that $(g_j^{-1}g_i)^{g^{-1}}\in G(R)$, or $g_j^{-1}g_i\in G(R)^g$. Therefore $g_iG(R)\cap G(R)^g=g_jG(R)\cap G(R)^g$, a contradiction.

Finally, let $H=\cup g_i G(\hR)\cap G(\hR)^g$. The set $H$ contains $G(R)$ and is closed, and so it contains the closure of $G(R)$, which is equal to $G^{sc}(\hR)$. Here $G^{sc}$ is the simply connected cover of the algebraic group $G$. Since $G(\hR)/G^{sc}(\hR)=\prod_w G(R_w)/G^{sc}(R_w)$, it is enough to show that $H$ projects onto a dense subset of $\prod_w G(R_w)/G^{sc}(R_w)$. For every $w$, the set $H$ contains $G(R_w)\cap G(R_w)^g$. If $G$ does not split at $w$, then $G(R_w)=G(k_w)$, and so $G(R_w)\cap G(R_w)^g=G(R_w)$. If $G$ splits at $w$, then, by using Cartan decomposition, we can assume that $g=\bmat \pi^n&0\\0&1\emat$, where $\pi$ is a uniformizer. In this case, the matrix $\bmat \al&\\&1\emat$ is contained in $G(R_w)\cap G(R_w)^g$ for every $\al$. Taking $\al\in R_w\setminus R_w^2$, we get generators for $G(R_w)/G^{sc}(R_w)$. Similarly, by using the Chinese Reminder Theorem, one can show that for every finite set of primes $S$, the set $H$ projects onto $\prod_{w\in S}G(R_w)/G^{sc}(R_w)$.
\end{proof}

By the Lemma, the inclusion of $G(k)$ in the set $G(k)G(\hR):=\{gh:\ g\in G(k),\ h\in G(\hR)\}$ induces a map $\Comm_n(G(R),G(k))\to\Comm_n(G(\hR),G(k)G(\hR))$. Since $g\in G(k)$ normalizes $G(\hR)$ if and only if $\chi_{G(R),G(k)}(g)=\chi_{G(\hR),G(k)G(\hR)}(g)=1$, if and only if $g$ normalizes $G(R)$, we get a bijection between $N_{G(k)}G(R)\bs\Comm_n(G(R),G(k))$ and $N_{G(k)G(\hR)}G(\hR)\bs\Comm_n(G(\hR),G(k)G(\hR))$, which shows that
$$\al_{G(\hR),G(k)G(\hR)}=\al_{G(R),G(k)}.$$

\begin{prop}\label{prop:R.k.hR.bA} The abscissae of convergence of $\ze_{G(R),G(k_v)}(s)$ and of $\ze_{G(\hR),G(\bA)}(s)$ are equal. Moreover, there is a constant $D$ such that for all $n$,
\begin{equation}\label{eq:c.leq.ineq}
c_{\leq n}(G(R),G(k_v))\leq c_{\leq n}(G(\hR),G(\bA))\leq Dc_{\leq Dn}(G(R),G(k_v))
\end{equation}
\end{prop}

\begin{proof} The set $G(\hR)\bs G(\bA)/G(k)$ is finite by \cite{P}; let $G(\hR)\ga_i G(k)$ be coset representatives. It follows that
\[
\ze_{G(\hR),G(\bA)}(s)=\sum \ze_{G(\hR),G(\hR)\ga_i G(k)}(s),
\]
and, hence, that $\al_{G(\hR),G(\bA)}=\max \al_{G(\hR),G(\hR)\ga_i G(k)}$. In particular, taking the trivial coset ($\ga_i=1$), we get that $\al_{G(\hR),G(\bA)}\geq\al_{G(\hR),G(\hR)G(k)}=\al_{G(R),G(k)}$. The left inequality of (\ref{eq:c.leq.ineq}) also follows.

For every $i$, the map $G(k)\to G(\hR)\ga_i G(k)$ given by $g\mapsto G(\hR)\ga_i g$ is a surjection with fibers $G(k)\cap G(\hR)^{\ga_i^{-1}}$. Let $\De=G(R)\cap G(\hR)^{\ga_i^{-1}}$, and denote the index of $\De$ in $G(R)$ by $N$. Then
\[
\ze_{G(\hR),G(k)\ga_i G(\hR)}(s)=\sum_{g\in G(\hR)\bs G(\hR)\ga_i G(k)}\chi(g)^{-s}\leq\sum_{g\in \De\bs G(k)}\chi(g\ga_i)^{-s}\leq
\]
by Lemma \ref{lem:chi.product},
\[
\leq \sum_{g\in \De\bs G(k)}\chi(g)^{-s}\chi(\ga_i)^s\leq N\chi(\ga_i)^s\sum_{g\in G(R)\bs G(k)}\chi(g)^{-s}=N\chi(\ga_i)^s\ze_{G(R),G(k)}.
\]
Therefore, the second inequality of (\ref{eq:c.leq.ineq}) holds, and $\al_{G(R),G(k)}=\al_{G(\hR),G(\bA)}$.
\end{proof}

By Lemma \ref{lem:Comm.prod} we have $\ze_{G(\hR),G(\bA)}(s)=\prod_w\ze_{G(R_w),G(k_w)}(s)$, the product being taken over the set of non-archimedian valuations of $k$ that are different from $v$. We turn to study the local zeta functions. If $G$ does not split over $w$ then $\ze_{G(R_w),G(k_w)}(s)=1$. Otherwise, it is $\ze_{\PGL_2(R_w),\PGL_2(k_w)}(s)$.

Let $K=\PGL_2(R_w)$. Choose a uniformizer, $\pi$, for $k_w$, and let $t=\bmat \pi&0\\0&1\emat$. It is known (Cartan decomposition) that every element of $\PGL_2(k_w)$ can be written as $g=k_1t^nk_2$ for $k_1,k_2\in K$ and $n\geq0$. For such an element, $g$, we have
\[
K^g\cap K=K^{k_1t^nk_2}\cap K=K^{t^nk_2}\cap K=\left(K^{t^n}\cap K\right)^{k_2},
\]
and in particular
\[
[K:K^g\cap K]=[K:K^{t^n}\cap K].
\]
Computing, we find that
\[
K^{t^n}\cap K=\left\{\bmat a&b\\c&d\emat : w(c)\geq n\right\}.
\]

We denote the size of the residue field $R/\pi$ by $|w|$.

\begin{lem} If $n>0$, then the index of $K^{t^n}\cap K$ in $K$ is $(|w|+1)|w|^{n-1}$.
\end{lem}

\begin{proof} Let $K_n$ be the $n$'th congruence subgroup. Then $K_n\subset K^{t^n}\cap K$, so it is enough to compute the index of the projection to $K/K_n=\PGL_2(R/\pi^n)$. The size of $\GL_2(R/\pi)$ is $(|w|^2-1)(|w|^2-|w|)$. Therefore the size of $\PGL_2(R/p)$ is $(|w|^2-1)|w|$. Therefore the size of $\PGL_2(R/\pi^n)$ is $(|w|^2-1)|w|^{1+3(n-1)}=(|w|^2-1)|w|^{3n-2}$. The projection of $K^{t^n}\cap K$ is the quotient of the group of upper triangular invertible matrices (there are $|(R/\pi^n)^\times|^2|R/\pi^n|$ of those) by the group of scalar matrices (there are $|(R/\pi^n)^\times|$ of those). Hence the projection of $K^{t^n}\cap K$ has size $(|w|-1)|w|^{2n-1}$.
\end{proof}

\begin{lem} If $n>0$, then the number of cosets of $K$ inside $Kt^nK$ is $(|w|+1)|w|^{n-1}$.
\end{lem}

\begin{proof} $K$ acts transitively on this set of cosets. The stabilizer of the coset $t^nK$ is exactly $K^{t^{-n}}\cap K$. But $t^{-n}$ is conjugate to $t^n$ by an element of $K$ (namely $\bmat0&1\\1&0\emat$). So the size of the orbit is the index of $K^{t^n}\cap K$ in $K$.
\end{proof}

%\nir{a more geometric proof is to show the second lemma from the fact that $G/K$ is a $p+1$ regular tree, and deduce the first from the second lemma}

Therefore
\begin{eqnarray}\label{eq:loc.zeta.func.PGL_2}
Z_{\PGL_2(R_w),\PGL_2(k_w)}(s)=1+\sum_{n=1}^\infty\left[(|w|+1)|w|^{n-1}\right]^{1-s}=
\end{eqnarray}
\begin{eqnarray*}
1+(|w|+1)^{1-s}\sum_{n=0}^\infty |w|^{(1-s)n}=1+(|w|+1)^{1-s}\frac{1}{1-|w|^{1-s}}.
\end{eqnarray*}

\begin{thm} Let $F$ be a local field, and let $\Ga$ be a lattice in $\PGL_2(F)$.
\begin{enumerate}
\item If $\Ga$ is arithmetic, then the sequence $c_{\leq n}(\Ga,\PGL_2(F))/n^2$ is bounded away from 0 and infinity, as $n$ tends to infinity.
\item If $\Ga$ is not arithmetic, then the sequence $c_{\leq n}(\Ga,\PGL_2(F))$ is bounded.
\end{enumerate}
\end{thm}

\begin{proof} By a theorem of Margulis \cite[Theorem IX.1.9(B)]{Mbook}, if $\Ga$ is not arithmetic then it has finite index in its commensurizer. Hence its commensurizer growth is bounded. As for the first claim, we first show it for the form $G=\PGL_2$. By Equation (\ref{eq:loc.zeta.func.PGL_2}),
\[
Z_{\PGL_2(\hR),\PGL_2(\bA)}(s)=\prod_w \left(1+\frac{(|w|+1)^{1-s}}{1-|w|^{1-s}}\right)=\zeta_k(s-1)\prod_w\left(1+(|w|+1)^{1-s}-|w|^{1-s}\right),
\]
where $\zeta_k(s)$ is the Dedekind zeta function of $k$. Assume that $s>1$. Since for all but finitely many $w$'s
\[
\frac{1}{2}(s-1)|w|^{-s}<(|w|+1)^{1-s}-|w|^{1-s}<2(s-1)|w|^{-s},
\]
and since
\[
\prod_w\left(1+\frac{1}{2}(s-1)|w|^{-s}\right)\textrm{\quad and \quad}\prod_w\left(1+2(s-1)|w|^{-s}\right)
\]
converge absolutely and are different from 0, we get that
\[
\prod_w\left(1+(|w|+1)^{1-s}-|w|^{1-s}\right)
\]
converges absolutely for $s>1$, and is non-zero. Therefore $Z_{\PGL_2(\hR),\PGL_2(\bA)}(s)$ is meromorphic in the half plane $\Re(s)>1$ and has a simple pole at $s=2$. By a Tauberian theorem, there is a constant $C>0$ such that
\[
\lim_{n\to\infty}\frac{c_{\leq n}(\PGL_2(\hR),\PGL_2(\bA))}{n^2}=C.
\]
The claim now follows by Proposition \ref{prop:R.k.hR.bA}.

To show the claim for the form $G=\PGL_1(D)$, first we note that
$$Z_{\PGL_2(\hR),\PGL_2(\bA)}/Z_{\PGL_1(D)(\hR),\PGL_1(D)(\bA)}$$
 is a finite product of functions of the form  (\ref{eq:loc.zeta.func.PGL_2}).
Indeed, as any quaternion algebra splits over all but finitely many primes only finitely many factors survive in the enumerator and denominator of the above quotient. If
$\PGL_1(D)$ does not split at $w$, then $\PGL_1(D)(R)=\PGL_1(D)(k_w)$ hence the corresponding factor in the denominator equals $1$.

We conclude that the biggest pole of $Z_{\PGL_2(\hR),\PGL_2(\bA)}/Z_{\PGL_1(D)(\hR),\PGL_1(D)(\bA)}$ is at $s=1$. Hence the biggest pole of $Z_{\PGL_1(D)(\hR),\PGL_1(D)(\bA)}$ is at $s=2$, it has meromorphic continuation to $\Re(s)>1$, and the pole at $s=2$ is simple. The rest of the argument is the same.
\end{proof}

%%%%%%%%%%%%%%%%%%%%%%%%%%%%%%%%%%%%%%%%%%%%%%%%%%%%%

%%%%%%%%%%%%% Comm growth of tree lattices %%%%%%%%%%%%%%%%%%%%%%

%%%%%%%%%%%%%%%%%%%%%%%%%%%%%%%%%%%%%%%%%%%%%%%

\section{Uniform Lattices in the Automorphism Group of Trees}\label{sec:countRecolorings}

As mentioned in the introduction, lattices in some Lie groups over a non-archimedean local field can be considered as lattices in the automorphism group of a tree. In this section we show that the commensurizer growth of tree lattices are much bigger than that of the corresponding lattices in Lie groups. We will consider only uniform tree lattices. For those lattices, it is well-known that the lattice is of finite index in its normalizer \cite{BL}, thus we will count the commensurizer growth up to the lattice instead of its normalizer.

In order to compute the commensurizer growth we will give a generalization and a refinement of the correspondence in \cite{LMZ} between the commensurizer and the set of recolorings of pointed graphs, looked at from a slightly different point of view.

Let $T$ be a uniform tree and let $G=\Aut(T)$. Suppose that $\Ga\subset G$ is a torsion-free lattice. Let $Y=\Ga\bs T$ be the quotient graph and let $\pi:T\to Y$ be the canonical projection. Fix a vertex $t_0\in T$ and let $\pi(t_0)=y_0\in Y$.

\begin{lem}\label{lem:phi}
The map $\phi: \Ga g\mapsto \pi\circ g$ is a bijection between $\Ga\bs G$ and the collection of covering maps $T\to Y$.
\end{lem}

\begin{proof}
We describe the inverse map. Suppose that $f:T\to Y$ is a covering. Since $T$ is simply connected, $f$ lifts to a map $g: T\to T$ such that $f=\pi\circ g$. Since $f=\pi\circ g$, $g$ is a covering map. Since $T$ is simply connected, $g$ is an automorphism.
\end{proof}

This map clearly intertwines the right multiplication by $g\in G$ on $\Ga \bs G$ and pre-composition by $g\in G$ on the set of covering maps $T\to Y$.

\begin{lem}\label{lem:phi*}
The map $\phi$ from Lemma \ref{lem:phi} induces a bijection between $\Ga\bs\Comm_n(\Ga,G)$ and the collection of covering maps $T\to Y$ whose $\Ga$-orbits have size $n$.
\end{lem}
\begin{proof}
$\Ga$ acts on the right on the cosets $\Ga g$ of $\Ga\bs G$. The stabilizer of $\Ga g$ under this action is $\Stab_{\Ga}(\Ga g)= \Ga \cap \Ga^g$. Thus, the $\Ga$-orbit of $\Ga g$ has cardinality $[\Ga: \Ga \cap \Ga^g]$.
Assume now that $[\Ga: \Ga \cap \Ga^g]=n < \infty$. As $\Ga$ is a lattice, by Lemma \ref{lem:chi.product}(1) $[\Ga: \Ga \cap \Ga^g]=[\Ga^g: \Ga \cap \Ga^g]$, hence $g\in \Comm_n(\Ga,G)$. Thus, $\Ga\bs\Comm_n$ is the set of cosets whose $\Ga$-orbit has size $n$, and by Lemma \ref{lem:phi} the assertion follows.
\end{proof}

We now find another realization of this set.
Given a cover $\sigma:T\to Y$, whose stabilizer $\De=\Stab_\Ga(\sigma)$ has finite index in $\Ga$, we get two covering maps from $\De\bs T$ to $Y$. One is the map induced by the inclusion $\De\subset\Ga$ (i.e. by the canonical projection), and the other is the map induced by $\sigma$. The graph $\De\bs T$ has a distinguished vertex---the image of $t_0$ under the quotient---and the first of the two covering maps sends this vertex to $y_0$.

\begin{defn} \begin{enumerate}
\item A quadruple $(X,x_0,f_1,f_2)$, where $X$ is a finite graph, $x_0$ is a vertex in $X$,  and $f_1,f_2:X\to Y$ are covering maps, such that $f_1(x_0)=y_0$, is called a twin cover (of the pair $(Y,y_0)$). The common degree of the covers (which is equal to $|X|/|Y|$) is called the degree of the twin cover.
\item A morphism between two twin covers $(X,x_0,f_1,f_2)$ and $(Z,z_0,h_1,h_2)$ is a covering $\phi:X\to Z$ such that $\phi(x_0)=z_0$, $f_1=h_1\circ\phi$, and $f_2=h_2\circ\phi$. If the degree of $\phi$ is not $|X|/|Y|$ we say that
$(X,x_0,f_1,f_2)$ factors through $(Z,z_0,h_1,h_2)$ via $\phi$.
\item A twin cover $(X,x_0,f_1,f_2)$ is called minimal if any morphism from it is either an isomorphism, or of degree $|X|/|Y|$.
\end{enumerate}
\end{defn}

By the last paragraph, we get a map $F$ from $\Ga\bs\Comm(G,\Ga)$ to the collection of twin covers of $(Y,y_0)$.

\begin{prop}\label{prop:Fbijection}
\begin{enumerate}
\item The image of $F$ is contained in the collection of minimal twin covers of $(Y,y_0)$.
\item $F$ gives a bijection between $\Ga\bs\Comm_n(G,\Ga)$ and isomorphism classes of minimal twin covers of degree $n$ of $(Y,y_0)$.
\end{enumerate}
\end{prop}

\begin{proof} \begin{enumerate}
\item If $\sigma:T\to Y$ is a cover, and $F(\sigma)=(X,x_0,f_1,f_2)$ factors through $(Z,z_0,h_1,h_2)$ via $\phi$, then the stabilizer of $\sigma$ contains $\pi_1(Z,z_0)\supset\pi_1(X,x_0)$, so $\pi_1(Z,z_0)=\pi_1(X,x_0)$, and $\phi$ must be an isomorphism.
\item The inverse map is constructed as follows: Given a twin cover $(X,x_0,f_1,f_2)$,
let $\tilde{X}$ be the universal cover of $X$; it is isomorphic to $T$.
Thus, there is a unique covering map $\xi:T\to X$ such that $\xi(t_0)=x_0$ and $f_1\circ \xi=\pi:T\to Y$. Define $\sigma:T\to Y$ to be the composition $f_2\circ \xi$.

We check that this map $M$ is indeed the inverse of $F$. To see that $M$ is well defined we need to check two things. First, if $(Z,z_0,h_1,h_2)$ is isomorphic to
$(X,x_0,f_1,f_2)$ via $\phi$ and $\mu: T\to Z$ is the unique cover such that $\mu(t_0)=z_0$ and $h_1\circ \mu=\pi$, we need to show that $f_2\circ \xi = h_2\circ \mu$.
Indeed $f_2\circ \xi = h_2\circ \phi \circ \xi = h_2\circ \mu$, where the first equality follows from the definition of $\phi$ and the second equality follows
from the uniqueness of $\mu$.
Second, we need to show that $\sigma$ has $\Ga$-orbit of size $|X|/|Y|$.
Let $\De\subset \pi_1(Y,y_0)$ be the image of $\pi_1(X,x_0)$ by $f_1$.
By the minimality of $(X,x_0,f_1,f_2)$, the $\Ga$-orbit of $\sigma$ has size  $|\De\bs T|/|\Ga\bs T|=|X|/|Y|$.
That $M\circ F=Id$ is clear. To see that $F\circ M=Id$ notice that by well-definedness of $M$, $\De=\Stab_{\Ga}(\sigma)$ and for $Z=\Stab_{\Ga}(\sigma)\bs T$, $z_0=\pi_Z(t_0)$ one gets an isomorphism $\phi:(X,x_0,f_1,f_2)\to (Z,z_0,i_*,\sigma_*)$.
\end{enumerate}
\end{proof}

Now we want to count commensurizer growth by counting the number of isomorphism classes of minimal twin covers of degree $n$ of $(Y,y_0)$, in order to prove Theorem \ref{thm:upper bound uniform lattice}.

Assume that in the uniform tree $T$ each vertex has degree $\geq 3$.
Let $c(n)=|\G\bs\Cnnp|$. Recall that any uniform tree lattice is of finite index in its normalizer, hence a good estimate on $c(n)$ will provide a good estimate on $c_{\leq n}(\Gamma,G)$.
Further, we will see that a good estimate on $c_{\leq n}(\Gamma',G)$, where $\Ga' \subset G$ is a torsion free uniform lattice
will
provide a good estimate on $c_{\leq n}(\Gamma,G)$ for \emph{any} uniform lattice $\Ga$ of $T$, thanks to the following result by Bass and Kulkarni \cite{BaKu}.
\begin{theo}\cite{BaKu}\label{thm:BassKulkarniTreeLattices}
Any two uniform lattices in $G=\Aut(T)$ are commensurable after conjugation.
\end{theo}

By Proposition~\ref{prop:Fbijection}, $c(n)$ is the number of isomorphism classes of minimal twin covers of $(Y,y_0)$ of degree $\leq n$.

\begin{proof}[Proof of Theorem \ref{thm:exact_bound}]
Let $\Ga'$ be a torsion free uniform lattice in $G=\Aut(T)$, $Y=\Ga'\bs T$ and $y_0=\pi(t_0)$ for the canonical projection $\pi: T\to \Ga'\bs T$.
Let $c_v, c_e$ be the numbers of vertices and edges in $Y$, respectively.
As  we count up to isomorphism classes, we can fix $VX=\{x_0=1,2,...,m c_v\}$ to be the vertex set for all covers $X$ of degree $m$ which we consider (in particular, $m\leq n$).

The number of base point preserving coverings of $(Y,y_0)$ of degree $m$ (call it pointed $m$-cover for short) equals the number of index $m$ subgroups of the fundamental group $\pi_1(Y,y_0)$.
%(e.g. \cite[Theorem 1.38]{HatcherAlgTop}).
This will count the number of possible $f_1$ in $(X,x_0,f_1,f_2)$.
Now $(Y,y_0)$ is homotopic to a bouquet of $k=k(\G')$ loops with $k\geq 3$ (it follows from the fact that each vertex in $T$, hence also in $Y$, has degree $\geq 3$), hence $\pi_1(Y,y_0)$ is isomorphic to the free group on $k$ generators, $F_k$.
Asymptotically, the number of index $m$ subgroups of $F_k$ is $m(m !)^{k-1}$, see e.g. \cite[Theorem 2.1]{SubgroupGrowthBook}.

First we prove the upper bound.
As a pointed covering, the base point $x$ in $f_2:(X,x)\to(Y,y_0)$ can take at most $mc_v$ values, and we use it to compute an upper bound on the number of possible $f_2$.
The discussion above yield that the number of twin covers (not necessarily minimal) of degree $m$ of $(Y,y_0)$ is
$\leq (1+\epsilon) m c_v (m(m !)^{k-1})^2$ (for any fixed $\epsilon>0$ and large enough $m$).

Thus, summing over all $1\leq m\leq n$ we obtain the upper bound
$$c(n) \leq n (1+\epsilon) n c_v (n(n !)^{k-1})^2 \leq 2^{c(\G')n \rm{lg}(n)}$$ for the  constant $c(\G')=2(k-1)(1+\epsilon)>0$. This proves the upper bound for torsion free lattices.

Let $\G$ be a uniform lattice.
By Theorem \ref{thm:BassKulkarniTreeLattices}, $\G$ is commensurable after conjugation to $\G'$, i.e. there is $g\in G$ such that $[\G': \G^g \cap \G']=t_1$ and $[\G^g: \G^g \cap \G']=t_2$ where $t_1,t_2$ are finite.
Clearly $|\G\bs\Cnn|=|\G^g\bs\Cnng|$ for every $n$.

By Proposition \ref{prop:IndexGapForCommensurales} we get

\begin{align*}
|\G\bs\Cnn|&\leq |(\G^g\cap \G')\bs\mathrm{Comm}_{\leq t_2n}(\Gamma^g \cap \G')| \\
&\leq |\G'\bs\mathrm{Comm}_{\leq t_1t_2n}(\G')|\leq 2^{(1+\epsilon)c(\G') t_1t_2 n\log(n)}
\end{align*}
(for any fixed $\epsilon>0$ and large enough $n$).
Hence $c_1(\G)=(1+\epsilon)c(\G')t_1t_2$ is a suitable constant.

Next, we prove the lower bound.
Choose an odd prime $p$ such that $\frac{n}{2}\leq p \leq n$, and use the estimate $c(n)\geq c(p)-c(p-1) := c'(p)$. Note that as $p$ is prime, any twin cover of degree $p$ is minimal.
We under-count the number of twin covers of degree $p$ by merely counting those of the form $(X,x_0,f,f_2)$, where $f$ is fixed and $f_2(x_0)=y_0$, and dividing by an upper bound on $|\Aut(X)|$, which is obviously an upper bound for the size of the isomorphism classes.

Let us estimate $|\Aut(X)|$. As $T$ is uniform, all vertex degrees in T are smaller than some constant $b$, hence also in $X$ (as $T\to X$ is a covering).  An automorphism $\phi$ of X sends a spanning tree to a spanning tree. $X$ has $pc_v$ vertices, thus the number of labeled spanning trees is at most $pc_v b^{pc_v-1}$, as after mapping $x_0$ somewhere, a neighbor of $x_0$ has at most $b$ options where to be mapped, and its neighbor has at most $b-1$ options, etc. Hence $|\Aut(X)|\leq pc_v b^{pc_v-1}$.

Thus,
$$c'(p)\geq (1-\epsilon')p(p !)^{k-1}\frac{1}{pc_v b^{pc_v-1}}\geq 2^{(1-\epsilon)(k-1)p\rm{lg}(p)}$$
(for any fixed $\epsilon>0$ and large enough $n$, hence large enough $p$).
Hence,
$c(n)\geq 2^{c'(\G')n \rm{lg}(n)}$
 where
$c'(\G')= \frac{(1-\epsilon)}{2}(k-1) >0$.
This proves the lower bound for torsion free lattices.

Let $\G$ be a uniform lattice and $t_1,t_2$ as before.
Then by Proposition \ref{prop:IndexGapForCommensurales}
$$|\G \backslash \Cnn|\geq 2^{c'(\G')\frac{1}{t_1t_2} n\lg(n)},$$
and as $[N_G\G : \G]<\infty$, $c_2(\G)=c'(\G')\frac{1}{t_1t_2} >0$ is a suitable constant.
\end{proof}
%%%%%%%%%%%%%%%%%%%%%%%%%%%%%%%%%%%%%%

%%%%%%% section 8 : Arbitrary growth %%%%%%%%%%%%%%%%%%%

%%%%%%%%%%%%%%%%%%%%%%%%%%%%%%%%%%%%%%%%%

\section{Arbitrary growth}\label{sec:arbitraryGrowth}
\begin{lem}\label{lem:ArbitraryGrowth}
Let $G=\GL_n(\bF_p)\ltimes\bF_p^n$ and $\bF_p\cong A\subset G$ be the subgroup of elements of the form $((*,0,\ldots,0),Id)$. Then
\[
Z_{G,A}(s)=1+\left(\frac{p^n-1}{p-1}-1\right)p^{-s}.
\]
\end{lem}

\begin{proof} The normalizer of $A$ in $G$ is $B\ltimes\bF_p^n$, where $B$ is the subgroup of $\GL_n(\bF_p)$ that stabilizes the line $(*,0,\ldots,0)$. Since the index of $B$ in $\GL_{n}(\bF_p)$ is $\frac{p^n-1}{p-1}$, and since for every element $g$ that is not in the normalizer $A^g\cap A$ is trivial, the result follows.
\end{proof}

The following proposition proves Theorem \ref{thm:manyGrowthFunc}.
\begin{prop} Let $A$ be the pro-cyclic group $\prod_p\bF_p$. For every function $f:\bN\to\bN$ there is a group $G$ containing $A$ such that $c_{\leq n}(A,G)\geq f(n)$, for every $n \in \bN$.
\end{prop}

\begin{proof} Choose a sequence $n_p$, indexed by the primes, such that $f(p+i)<\frac{p^{n_p}-1}{p-1}$ for every $0 \leq i \leq p$, and define $G$ to be the restricted product of $\GL_{n_p}(\bF_p)\ltimes\bF_p^{n_p}$ relative to $\bF_p$, for every prime $p$. For every $n \in \bN$, there exists a prime $p$ such that $n/2 \leq p \leq n$, thus we have
\[
c_{\leq n}(A,G) \geq c_{\leq p}(A,G)\geq c_{p}(A,G)\geq\frac{p^{n_p}-1}{p-1}-1\geq f(n)
\]
by the choice of the sequence $(n_p)$ and Lemma \ref{lem:ArbitraryGrowth}.
\end{proof}

\vspace{0.1 in}
{\it{Acknowledgement}}
We thank Alex Lubotzky and Shahar Mozes for helpful discussions.

%\eran{For some reason the first item in Bib is not ordered alphabetically.}
\bibliographystyle{amsplain}
\bibliography{seonhee}
\end{document}